\documentclass[english]{amsart}

\usepackage{ulem}
\usepackage{graphicx}

\usepackage{amsmath}

\usepackage{amssymb}

    \DeclareFontFamily{U}{wncy}{}
    \DeclareFontShape{U}{wncy}{m}{n}{<->wncyr10}{}
    \DeclareSymbolFont{mcy}{U}{wncy}{m}{n}
    \DeclareMathSymbol{\Sha}{\mathord}{mcy}{"58}

\usepackage{amscd}
\usepackage[all,cmtip,line]{xy}

\newlength{\ownl}

\newcommand{\Gal}{{\operatorname{Gal}\,}}

\newcommand{\Spec}{{\operatorname{Spec}\,}}

\newcommand{\GL}{\operatorname{GL}}

\newcommand{\SL}{\operatorname{SL}}

\newcommand{\loc}{{\operatorname{loc}}}

\newcommand{\C}{{\mathbb{C}}}

\newcommand{\F}{{\mathbb{F}}}

\newcommand{\Q}{{\mathbb{Q}}}

\newcommand{\Z}{{\mathbb{Z}}}

 \newcommand{\rhobar   }{{\overline{\rho}}}

\newcommand{\Qbar}{{\overline{\Q}}}

\def\RCS$#1: #2 ${\expandafter\def\csname RCS#1\endcsname{#2}}
\RCS$Revision: 327 $
\RCS$Date: 2010-11-15 13:46:16 -0600 (Mon, 15 Nov 2010) $

\newcommand{\cL}{\mathcal{L}}

\newcommand{\rbar}{\bar{r}}

\newcommand{\Qp}{\Q_p}

\newcommand{\Qpbar}{\overline{\Q}_p}
\newcommand{\Qlbar}{\overline{\Q}_{l}}
\newcommand{\Fpbar}{\overline{\F}_p}

\newcommand{\Flbar}{\overline{\F}_l}

\usepackage{amsthm}

\newtheorem{thm}{Theorem}[subsection]

\newtheorem{cor}[thm]{Corollary}

\newtheorem{prop}[thm]{Proposition}
\newtheorem{conj}[thm]{Conjecture} \theoremstyle{definition}
\newtheorem{conjecture}[thm]{Conjecture} \theoremstyle{definition}
 \theoremstyle{remark}
\newtheorem{rem}[thm]{Remark}

\newtheorem{question}[thm]{Question} 
\numberwithin{equation}{subsection}

\theoremstyle{definition}

\usepackage[bookmarksopen,bookmarksdepth=2,breaklinks=true]{hyperref}

\begin{document}
\title[]{Slopes of modular forms} 

\author{Kevin Buzzard} \email{k.buzzard@imperial.ac.uk} \address{Department of
  Mathematics, Imperial College London}

\author{Toby Gee} \email{toby.gee@imperial.ac.uk} \address{Department of
  Mathematics, Imperial College London}\thanks{Both authors were
  supported in part by EPSRC grant EP/L025485/1. The second author
  was additionally supported by  a Leverhulme Prize, Marie Curie Career
  Integration Grant 303605, and 
  ERC Starting Grant 306326.}
\begin{abstract}
  We survey the progress (or lack thereof!) that has been made on some questions
  about the $p$-adic slopes of modular forms that were raised by the first
  author in~\cite{MR2141701}, discuss strategies for making further progress,
  and examine other related questions.
\end{abstract}

\maketitle 
\section{Introduction}\label{sec: introduction}
\subsection{}The question of the distribution of the local components of
automorphic representations at finite places has received a great deal of
attention. 

In the case of fixing an automorphic representation and varying the
finite place, we now have the recently-proved Sato--Tate
conjecture for elliptic curves over totally real
fields~\cite{MR2630056,MR2470687,MR2470688}. 
More recently, there has been
much progress on questions where the automorphic representation varies, but the
finite place is fixed; see~\cite{MR3004076}, and the references discussed in its
introduction, for a detailed history of the question. Still more recently, there
has been the fascinating work of~\cite{shin2012sato} on hybrid problems, where
both the finite place and the automorphic representation are allowed to vary,
but we will have nothing to say about this here. 

In this survey we will consider some other variants of this basic question,
including $p$-adic ones. Just as in the classical setting, there are really
several questions here, which will have different answers depending on what is
varying: for example if one fixes a weight~2 modular form corresponding to a
non-CM elliptic curve, then it is ordinary for a density one set of primes; however if one
fixes a prime and a level and considers eigenforms of all weights, then almost
none of them are ordinary (the dimension of the ordinary part remains bounded by
Hida theory as the weight gets bigger).

We will for the most part limit ourselves to the case of classical modular forms
for several reasons. The questions we consider are already interesting (and
largely completely open) in this case, and in addition, there appear to be
interesting phenomena that we do not expect to generalise in any obvious way
(see Remark~\ref{rem: should there be slope conjectures any more generally?}
below.) However, it seems worth recording a natural question (from the point of
view of the $p$-adic Langlands program) about the distribution of local
parameters as the tame level varies; for concreteness, we phrase the question
for $\GL_n$ over a CM field, but the same question could be asked in greater
generality in an obvious fashion.

Fix a CM field $F$, and consider regular algebraic essentially conjugate
self-dual cuspidal automorphic representations $\pi$ of $\GL_n/F$. Fix an
isomorphism between $\Qpbar$ and $\C$, and a place $v|p$ of $F$. Assume that
$\pi_v$ is unramified (one could instead consider $\pi_v$ lying on a particular
Bernstein component).  To such a $\pi$ is associated a Galois representation
$\rho_\pi:\Gal(\overline{F}/F)\to\GL_n(\Qpbar)$, and
$\rho_\pi|_{\Gal(\overline{F_v}/F_v)}$ is crystalline, with
Hodge--Tate weights determined by $\pi_\infty$. (See the introduction
to~\cite{MR3272052} for this result, and a discussion of the history of its proof. Thanks to the work of~\cite{harris2013rigid}
and~\cite{varmathesis}, the result is now known without the assumption of
essentially conjugate self-duality; but the cuspidal automorphic representations
of a fixed regular algebraic infinite type
which are \emph{not} essentially conjugate self-dual are expected to be rather sparse,
and in particular precise asymptotics for the number of such representations as
the level varies are unknown, 
and it therefore seems unwise to speculate about equidistribution questions for
them. Note that in the essentially conjugate self-dual case, these automorphic
representations arise via endoscopy from automorphic representations on unitary
groups which are discrete series at infinity, and can thus be counted by the
trace formula.) If we now run over~$\pi'$ of the same infinity type,
which have~$\pi'_v$
unramified, and which furthermore have $\rhobar_{\pi'}\cong\rhobar_{\pi}$ (the
bar denoting reduction to $\GL_n(\Fpbar)$), then the local representations
$\rho_{\pi'}|_{\Gal(\overline{F}_v/F_v)}$ naturally give rise to points of the
corresponding (framed) deformation ring for crystalline lifts of
$\rhobar_\pi|_{\Gal(\overline{F}_v/F_v)}$ of the given Hodge--Tate weights. The
existence of level-raising congruences mean that one can often prove that this (multi)set is 
infinite (and it is expected to always be infinite), and one could ask
whether some form of equidistribution of the
$\rho_{\pi'}|_{\Gal(\overline{F}_v/F_v)}$ holds in the rigid-analytic generic
fibre of the crystalline deformation ring.

Unfortunately, this appears to be a very hard problem. Indeed, we do not in
general even know that every irreducible component of the generic fibre
of the local deformation space contains even a
single $\rho_{\pi'}|_{\Gal(\overline{F}_v/F_v)}$; it is certainly expected that
this holds, and a positive solution would yield a huge improvement on the
existing automorphy lifting theorems (\emph{cf.}\ the introduction
to~\cite{CEGGPS}). Automorphy lifting theorems can sometimes be used to show
that if an irreducible component contains an automorphic point, then it contains
a Zariski-dense set of automorphic points, but they do not appear to be able to
say anything about $p$-adic density, or about possible equidistribution.

More generally, one could allow the weight (and, if one wishes, the level at
$p$) to vary (as well as, or instead of, allowing the level to vary) and ask
about equidistribution in the generic fibre of the full deformation ring, with no $p$-adic Hodge
theoretic conditions imposed. The points arising will necessarily lie on the
sublocus of crystalline (or more generally, if the level at $p$ varies,
potentially semistable) representations, but as these are expected to be
Zariski dense (indeed, this is known in most cases by the results
of~\cite{MR3037022} and~\cite{MR3235551}), it seems reasonable to conjecture that the points will also be
Zariski dense.

One could also consider the case of a place $v\nmid p$, where very similar
questions could be asked (except that there are no longer any $p$-adic
Hodge-theoretic conditions), and we are similarly ignorant (although the automorphy
lifting machinery can often be used to show that each irreducible component
contains an automorphic point, using the Khare--Wintenberger method~\cite[Thm.\ 3.3]{kw} and Taylor's
Ihara-avoidance result~\cite{MR2470688}; see~\cite[\S 5]{MR2785764}).

In the case of modular forms (over $\Q$) one can make all of this rather more
concrete, due to a pleasing low-dimensional coincidence: an irreducible
 two-dimensional crystalline representation of $\Gal(\Qpbar/\Qp)$ is almost
always completely determined by
its Hodge--Tate weights and the trace of the crystalline Frobenius (because
there is almost always a unique weakly admissible filtration on
the associated
filtered $\phi$-module -- see Section~\ref{subsec:centre of weight space} below). This means
that if we work with modular forms of weight $k$ and level prime to $p$, the
local $p$-adic Galois representation is almost always determined by the Hecke
eigenvalue $a_p$ (modulo the issue of semisimplicity in the ordinary case),
and the question above reduces to the question of studying the
$p$-adic behaviour of $a_p$. Such questions were studied computationally
(and independently) by Gouv\^ea and one of us (KB), for the most part in
level 1, when the era of computation of modular forms was in its infancy.
Gouv\^ea noticed (see the questions in \S2 of~\cite{MR1824885})
that in weight~$k$, the $p$-adic valuation $v_p(a_p)$ of~$a_p$ (normalised
so that $v_p(p)=1$) was almost always at most $(k-1)/(p+1)$, an observation
which at
the time did not appear to be predicted by any conjectures. Gouv\^ea
and Buzzard also noticed that $v_p(a_p)$ was almost always an integer,
an observation which even now is not particularly well-understood.
Furthermore, in level 1, the primes~$p$ for which there existed
forms with $v_p(a_p)>(k-1)/(p+1)$ seemed to \emph{coincide} with the primes
for which there existed forms with $v_p(a_p)\not\in\Z$. These led
Buzzard in \S1 of~\cite{MR2141701}
to formulate the notion of an $\SL_2(\Z)$-irregular prime, a prime
for which there exists a level~1 non-ordinary eigenform of weight
at most~$p+1$. Indeed one might even wonder whether the following are equivalent:

\begin{itemize}
\item $p$ is $\SL_2(\Z)$-irregular;
\item There exists a level 1 eigenform with $v_p(a_p)\not\in\Z$;
\item There exists a level 1 eigenform of weight~$k$
with $v_p(a_p)>(k-1)/(p+1)$.
\end{itemize}

One can check whether a given prime $p$ is $\SL_2(\Z)$-regular
or not in finite time (one just needs to compute the determinant
of the action of $T_p$ on level 1 modular forms of weight $k$
for each $k\leq p+1$ and check if it is always a $p$-adic unit; in fact
one only has to check cusp forms of weights $4\leq k\leq (p+3)/2$
because of known results
about $\theta$-cyles); one can also verify with machine computations
that the second or third
conditions hold by exhibiting an explicit eigenform with
the property in question. The authors do not know how to verify
with machine computations that the second or third conditions fail;
equivalently, how to prove for a given $p$ either that all $T_p$-eigenvalues
$a_p$ of all level~1 forms of all weights have integral $p$-adic
valuations, or that they all satisfy $v_p(a_p)\leq (k-1)/(p+1)$.
In particular it is still logically possible that for \emph{every}
prime number there will be some level~1 eigenforms
satifying $v_p(a_p)\not\in\Z$ or $v_p(a_p)>(k-1)/(p+1)$.
However this seems very unlikely -- for example $p=2$ is an $\SL_2(\Z)$-regular
prime, and the first author has
computed $v_p(a_p)$ for $p=2$ and for all $k\leq 2048$ and has found
no examples where $v_2(a_2)\not\in\Z$ or $v_2(a_2)>(k-1)/3$.
Gouvea also
made substantial calculations for all other $p<100$ which add further
weight to the idea that the conditions are equivalent.

There are precisely two $\SL_2(\Z)$-irregular primes less than 100, namely
59 and~79, and it does not appear to be known whether there are infinitely many
$\SL_2(\Z)$-regular primes or whether there are infinitely
many $\SL_2(\Z)$-irregular primes.
(However, Frank Calegari has
given~\url{https://galoisrepresentations.wordpress.com/2015/03/03/review-of-buzzard-gee/} an argument which shows that under standard conjectures about the
existence of prime values of polynomials with rational coefficients,
then there are infinitely many $\SL_2(\Z)$-irregular primes.) Note that for $p=59$ and $p=79$
eigenforms with $v_p(a_p)\not\in\Z$ and $v_p(a_p)>(k-1)/(p+1)$
do exist, but any given eigenform will typically satisfy
at most one of these conditions,
and we do not even know how to show that the second and third conditions are equivalent. 

Buzzard conjectured that for
an $\SL_2(\Z)$-regular prime, $v_p(a_p)$ was integral for all
level~1 eigenforms, and even conjectured an algorithm to compute
these valuations in all weights. Similar conjectures were made
at more general levels~$N>1$ prime to~$p$, and indeed Buzzard formulated
the notion of a $\Gamma_0(N)$-regular prime -- for $p>2$ this
is a prime $p\nmid N$ such that all eigenforms of level $\Gamma_0(N)$
and weight at most $p+1$ are ordinary, although here one has to
be a little more careful when $p=2$ (and even for $p>2$ some
care needs to be taken
when generalising this notion to $\Gamma_1(N)$ because allowing odd weights
complicates the picture somewhat; see Remark~\ref{k-odd}.)

These observations of Buzzard and Gouv\^ea can be thought of
as saying something about the behaviour of the Coleman--Mazur eigencurve near
the centre of weight space. Results of Buzzard--Kilford~\cite{MR2135280}, Roe~\cite{roe:3adic}, Wan--Xiao--Zhang~\cite{1407.0279} and Liu--Wan--Xiao~\cite{1412.2584} indicate that there is even more structure near
the boundary of weight space; 
this structure translates into
concrete assertions about $v_p(a_p)$ when $a_p$ is the $U_p$-eigenvalue
of a newform of level $\Gamma_1(Np^r)$ and character of conductor $Mp^r$
for some $M\mid N$ coprime to~$p$. We make precise conjectures
in Section~\ref{subsec:boundary of weight space}. On the other hand, perhaps these results are intimately
related to the $p$-adic Hodge-theoretic coincidence alluded to above -- that
in this low-dimensional situation there is usually only one (up to isomorphism)
weakly admissible filtration on the Weil--Deligne representation in
question. In particular such structure might not be so easily
found in a general unitary group eigenvariety.

Having formulated these conjectures, in Section~\ref{sec: my heuristic} we
discuss a potential approach to them via modularity lifting theorems.

\subsection{Acknowledgements}We would like to thank Christophe Breuil, Frank
Calegari and Matthew Emerton for many helpful conversations over the past 14 years about the subjects of this survey. We would like to thank Vytautas
Pa{\v{s}}k{\=u}nas for a helpful conversation about the Breuil--M\'ezard
conjecture, and Liang Xiao for explaining some aspects of the
papers~\cite{1407.0279} and~\cite{1412.2584} to us. We would like to thank Frank Calegari, Matthew Emerton and James
Newton for their helpful comments on an earlier draft of this paper;
in addition, we would like to thank Frank Calegari for refereeing the
paper (\url{https://galoisrepresentations.wordpress.com/2015/03/03/review-of-buzzard-gee/})
and for making several helpful comments in doing so. The second author would like to thank the organisers of the Simons
Symposium on Automorphic Forms and the Trace Formula for the invitation to speak
at the symposium, and for the suggestion of turning the rough notes for the talk
into this survey. 
\section{Limiting distributions of eigenvalues}\label{sec: limiting
  distributions} In this section we briefly review some conjectures and questions about
the limiting distributions of eigenvalues of Hecke operators in the $p$-adic
context. These questions will not be the main focus of our discussions, but as
they are perhaps the most natural analogues of the questions considered
in~\cite{shin2012sato}, it seems worth recording them.

\subsection{$\ell=p$: Conjectures of Gouv\^ea} 
The reference for this section is the paper~\cite{MR1824885}.
Fix a prime $p$, an integer
$N\ge 1$ coprime to~$p$, and consider the operator $U_p$
on the spaces of classical modular forms $S_k(\Gamma_0(Np))$ for varying
weights $k\ge 2$. 
The characteristic polynomial of $U_p$ has integer coefficients so it
makes sense to consider the slopes of the eigenvalues --
by definition, these are the $p$-adic valuations of the eigenvalues
considered as elements of $\Qpbar$. The eigenvalues themselves fall
into two categories. The ones corresponding to eigenforms
which are new at~$p$ (corresponding to Steinberg representations) have
$U_p$-eigenvalues $\pm p^{(k-2)/2}$, and thus slope $(k-2)/2$.
The other eigenvalues come in pairs, each pair being associated
to an eigenvalue of $T_p$ on $S_k(\Gamma_0(N))$, and if the
$T_p$-eigenvalue is $a_p$ (considered as an element of $\Qpbar$)
then the corresponding two $p$-oldforms have eigenvalues given by the
roots of $x^2-a_px+p^{k-1}$; so the slopes $\alpha,\beta\in[0,k-1]$ satisfy
$\alpha+\beta=k-1$. Note that
$\min\{\alpha,\beta\}=\min\{v_p(a_p),\frac{k-1}{2}\}$ by the theory of the Newton polygon, and in particular if $v_p(a_p)<\frac{k-1}{2}$
then $v_p(a_p)$ can be read off from $\alpha$ and $\beta$.

Now consider the (multi-)set of slopes of $p$-oldforms, normalised by dividing
by $k-1$ to lie
in the range $[0,1]$. More precisely we could consider the measure (a finite
sum of point measures, normalised to have total mass~1) attached to this
multiset in weight~$k$. Let $k$ tend to $\infty$ and consider how these
measures vary. Is there a limiting measure?
\begin{conj}\label{limiting distribution}
  (Gouv\^ea) The slopes converge to the measure which is uniform on
  $[0,\frac{1}{p+1}]\cup[\frac{p}{p+1},1]$ and $0$ elsewhere.
\end{conj}
This is supported by the computational evidence, which is particularly convincing
in the $\Gamma_0(N)$-regular case.
This conjecture implies
that if $a_p$ runs through the eigenvalues of $T_p$ on $S_k(\Gamma_0(N))$
then we ``usually'' have $v(a_p)\le (k-1)/(p+1)$. 
This appears to be the case, although the reasons why are not well
understood. If $p>2$ is $\Gamma_0(N)$-regular however,
the (purely local -- see below) main result of~\cite{MR2060368} shows
that $v(a_p)\le \lfloor (k-2)/(p-1)\rfloor$. One might hope that the main result of~\cite{MR2060368} could be strengthened
to show that in fact $v(a_p)\le\frac{k-1}{p+1}$; it seems likely that the
required local statement is true, but Berger tells us
that the proof
in~\cite{MR2060368} does not seem to extend to this more general range. (This
problem is carefully examined in Mathieu Vienney's unpublished PhD thesis.)

\subsection{$\ell\ne p$}\label{subsec:distribution of eigenvalues for l not
  p}
In the previous subsection we talked about the distribution of $a_p$,
the eigenvalues of $T_p$ on $S_k(\Gamma_0(N))$, considered as elements
of $\Qpbar$. 
The Ramanujan bounds and the Sato--Tate conjecture give us information
about the eigenvalues of $T_p$ as elements of the complex numbers. What
about the behaviour of the $a_p$ as elements of $\Qlbar$ for $\ell\not=p$
prime? We have very little idea what to expect. In this short section we
merely present a sample of some computational results concerning the even weaker
question of the distribution of the reductions of the $a_p$ as elements of
$\Flbar$. In contrast to the previous section we here vary $N$ and
keep~$k=2$ fixed. 
More precisely, we fix distinct $\ell$ and $p$, and then loop over $N\geq1$
coprime to $\ell p$ and compute the eigenvalues $\overline{a}_p$ of $T_p$
acting on $S_2(\Gamma_0(N);\Flbar)$. Here is a sample of the results
with $p=5$ and $\ell=3$, looping over the first 5,533,155 newforms:

\begin{tabular}{|c|c|}
\hline
Size of $\F_3[\overline{a}_5]$&Number\\
\hline
$3^1$&80656\\
$3^2$&38738\\
$3^3$&35880\\
$3^4$&32968\\
$3^5$&35330\\
$3^6$&33372\\
$3^7$&34601\\
$3^8$&33896\\
$3^9$&35262\\
$3^{10}$&33600\\
\hline
\end{tabular}

The first numbers in the second column of this the table are \emph{not}
decreasing, which is perhaps not what one might initially guess; Frank
Calegari observed that this could perhaps be explained by observing that
if you choose a random element of a finite field $\F_q$ then the field
it generates over $\F_p$ might be strictly smaller than $\F_q$,
and the heuristics are perhaps complicated
by this.
  \section{The Gouv\^ea--Mazur Conjecture/Buzzard's Conjectures}\label{sec: KB conjectures}
  \subsection{} Coleman theory (see Theorem~D of~\cite{MR1431135}) tells us that for a fixed prime $p$ and tame level~$N$,
there is a function $M(n)$ such
  that if $k_1,k_2>n+1$ and $k_1\equiv k_2\pmod{p^{M(n)}(p-1)}$, then the sequences of slopes
  (with multiplicities) of classical modular forms of level $Np$
and weights $k_1,k_2$ agree up to slope $n$. A more geometric
way to think about this theorem is that given a point on the eigencurve
of slope $\alpha\leq n$, there is a small neighbourhood of that point in
the eigencurve, which maps in a finite manner down to a disc in weight
space of some explicit radius $p^{-M(n)}$ and such that all the points
in the neighbourhood have slope $\alpha$.
Gouv\^ea and
  Mazur~\cite{MR1122070} conjectured that we could take $M(n)=n$; for $n=0$,
  this is a theorem of Hida (his ordinary families are finite over entire
components of weight space). Wan~\cite{MR1632794} deduced from Coleman's results
  that $M(n)$ could be taken to be quadratic (with the implicit constants
  depending on both~$p$ and~$N$; as far as we know, it is still an open problem to
  obtain a quadratic bound independent of either $p$ or $N$). However, Buzzard
  and Calegari~\cite{MR2059481} found an explicit counterexample to the
  conjecture that $M(n)=n$ always works.

On the other hand, Buzzard~\cite{MR2141701} accumulated a lot of numerical
evidence that whenever $p$ is $\Gamma_0(N)$-regular, many (but not all)
families of eigenforms seemed to have slopes  which were locally
equal to~$n$ on discs of size $p^{-L(n)}$ with $L(n)$ seemingly linear
in $\log(n)$ -- a much stronger bound than the Gouv\^ea--Mazur conjectures.
For example if $p=2$, $N=1$ then the classical slopes at weight $k=2^d$
(the largest of which is approximately $k/3$) seem to be an initial segment
of the classical slopes at weight $2^{d+1}$. For example, the 2-adic slopes
in level~1 and weight~$128=2^7$ are
$$3, 7, 13, 15, 17, 25, 29, 31, 33, 37$$
and the Gouv\^ea--Mazur conjectures would predict that 
the slopes which were at most 7 should show up in weight~$256=2^8$. However
in weight~256 the slopes are
$$3, 7, 13, 15, 17, 25, 29, 31, 33, 37, 47, 49, 51,\ldots$$
and more generally the slopes at weight equal to a power of~2 all seem
to be initial segments of the infinite slope sequence on overconvergent
2-adic forms of weight~0, a sequence explicitly computed in Corollary~1
of~\cite{MR2135279}. In particular, if one
were to restrict to $p=2$, $N=1$ and $k$ a power of~2 then $M(n)$ can
be conjecturally taken to be the base~2 logarithm of $3n$. Note
also that the counterexamples at level $\Gamma_0(N)$
to the Gouv\^ea--Mazur conjecture in~\cite{MR2059481}
were all $\Gamma_0(N)$-irregular. It may well be the case that
the Gouv\^ea--Mazur conjectures are true at level $\Gamma_0(N)$
if one restricts to $\Gamma_0(N)$-regular primes -- indeed the numerical
examples above initially seem to lend credence to the hope that something
an order of magnitude stronger than the Gouv\^ea--Mazur conjectures
might be true in the $\Gamma_0(N)$-regular case. However life is not quite
so easy -- numerical evidence seems to indicate that near to a newform for $\Gamma_0(Np)$
on the eigencurve, the behaviour of slopes seems to be broadly speaking
behaving in the same sort of way
as predicted by the Gouv\^ea--Mazur conjectures. For example, again with $p=2$
and $N=1$, computer calculations give that the slopes in weight
$38+2^8$ are
$$5, 8, 16, 18, 18, 20, 29, 32, 37, 40, 45, 50, 50, 56, 61, 64, 70,\ldots$$
whereas in weight $38+2^9$ they are
$$5, 8, 17, 18, 18, 19, 29, 32, 37, 40, 45, 50, 50, 56, 61, 64, 70,\ldots.$$
Again one sees evidence of something far stronger than the Gouv\^ea--Mazur
conjecture going on (the Gouv\^ea--Mazur conjecture only predicts equality
of slopes which are at most~8); however there seems to be a family
which has slope~16 in weight~$38+2^8$ and slope~17 in weight $38+2^9$.
This family could well be passing through a classical newform of level
$\Gamma_0(2)$
in weight~38, and newforms in weight~38 have slope $(38-2)/2=18$, so
one sees that for just this one family $M(n)$ is behaving much more like
something linear in~$n$. 

Staying in the $\Gamma_0(N)$-regular case,
Buzzard found a lot of evidence for a far more precise conjecture
than the Gouv\^ea--Mazur conjecture -- one that gives a complete
description of the slopes in the $\Gamma_0(N)$-regular case, in terms of a
recursive algorithm, which is purely combinatorial in nature and uses
nothing about modular forms at all.
Then (see~\cite[\S 3]{MR2141701} for a more detailed discussion) the
algorithm can for the most part be deduced from various heuristic
assumptions about families of $p$-adic
modular forms, for example the very strong ``logarithmic''
form of the Gouv\^ea--Mazur conjecture mentioned above, plus some heuristics
about behaviour of slopes near newforms that seem hard to justify.
Unfortunately, essentially nothing is known about these
conjectures, even in the simplest case $N=1$ and $p=2$, where the slopes
are all conjectured to be integers but even this is not known. 

In fact it does not even seem to be known that the original form of the
Gouv\^ea--Mazur conjecture (in the $\Gamma_0(N)$-regular case) is a consequence
of Buzzard's conjectures; see~\cite[Q.\ 4.11]{MR2141701}. It would also be of
interest to examine Buzzard's original data to try to formulate a precise
conjecture about the best possible value of $M(n)$ in the $\Gamma_0(N)$-regular
case. The following are
combinatorial questions, and are presumably accessible.

\begin{question} Say $p$ is $\Gamma_0(N)$-regular.
  \begin{enumerate}
  \item Does the Gouv\^ea--Mazur conjecture for $(p,N)$, or perhaps something even stronger, follow from Buzzard's conjectures? 
  \item Does Conjecture~\ref{limiting distribution} follow from Buzzard's conjectures?
  \end{enumerate}
\end{question}

One immediate consequence of Buzzard's conjectures is that in the $\Gamma_0(N)$-regular
case, all of the slopes should be integers. This can definitely fail in the $\Gamma_0(N)$-irregular case (and is a source of counterexamples to the
Gouv\^ea--Mazur conjecture), and we suspect that understanding this phenomenon
could be helpful in proving the full conjectures (see the discussion
in Section~\ref{sec: my heuristic} below). In Section~\ref{subsec:centre of weight space} we will
explain a purely local conjecture that would imply this integrality. 

Note that Lisa Clay's PhD thesis also studies this problem and makes
the observation that the combinatorial recipes seem to remain valid
when restricting to the subset of eigenforms with a fixed mod~$p$
Galois representation which is reducible locally at~$p$. 

\section{Local questions}\subsection{The centre of weight
  space}\label{subsec:centre of weight space}
In this section we discuss some purely
local conjectures and questions about $p$-adic Galois representations that are
motivated by the conjectures of Section~\ref{sec: KB conjectures}. We briefly
recall the relevant local Galois representations and their relationship to the
global picture, referring the reader to the introduction to~\cite{MR2511912} for
further details. If $k\ge 2$ and $a_p\in\Qpbar$ with $v(a_p)>0$, then there is a
two-dimensional crystalline representation $V_{k,a_p}$ with Hodge--Tate weights
$0,k-1$, with the property that the crystalline Frobenius of the corresponding
weakly admissible module has characteristic polynomial
$X^2-a_pX+p^{k-1}$. Furthermore, if $a_p^2\ne 4p^{k-1}$ then $V_{k,a_p}$ is
uniquely determined up to isomorphism. This is easily checked by directly
computing the possible Hodge filtrations on the weakly admissible module;
see for example~\cite[Prop.\ 2.4.5]{MR2642406}. This is a low-dimensional
coincidence however -- a certain parameter space of flags is connected
of dimension zero in this situation.

The relevance of this representation to the questions of Section~\ref{sec: KB
  conjectures} is that if $f\in S_k(\Gamma_0(N),\Qpbar)$ is an eigenform with $a_p^2\ne 4p^{k-1}$ (which is
expected to always hold; in the case $N=1$ it holds by Theorem 1 of~\cite{MR1824885}, and the paper \cite{MR1600034} proves
that it holds for general $N$ if $k=2$, and for general $k,N$ if one assumes the
Tate conjecture) then $\rho_f|_{\Gal(\Qpbar/\Qp)}\cong V_{k,a_p}$. 

As explained in~\cite[\S1]{MR2141701}, $p>2$ is $\Gamma_0(N)$-regular if and only if $\rhobar_f|_{\Gal(\Qpbar/\Qp)}$ is reducible for every $f\in S_k(\Gamma_0(N))$
(and every $k\ge 2$). This suggests that the problem of determining when $\overline{V}_{k,a_p}$ (the
reduction of $V_{k,a_p}$ modulo $p$) is reducible could be relevant to the
conjectures of Section~\ref{sec: KB conjectures}. To this end, we have the
following conjecture.
\begin{conj}\label{conj: integral slopes}
  If $p$ is odd, $k$ is even and $v(a_p)\notin\Z$  then
  $\overline{V}_{k,a_p}$ is irreducible.
\end{conj}

\begin{rem}
  \label{rem:this is consistent with Buzzard's conjectures}Any modular form of
  level $\Gamma_0(N)$ necessarily has even weight, and this conjecture would
  therefore imply for $p>2$
that in the $\Gamma_0(N)$-regular case, all slopes are
  integral, as Buzzard's conjectures predict (see \S3 above).\end{rem}
\begin{rem}
  \label{rem: history of the conjecture, it's false if you look more generally}
  This conjecture is arguably ``folklore'' but seems to originate
in emails between Breuil, Buzzard and Emerton in 2005.
\end{rem}
\begin{rem}\label{rem: ordinary case}The conjecture is of course false
  without the assumption that $v(a_p)\notin\Z$; indeed, if $v(a_p)=0$
  then we are in the ordinary case, and $V_{k,a_p}$ is reducible (and
  so $\overline{V}_{k,a_p}$ is certainly reducible).
\end{rem}
\begin{rem}\label{k-odd}
  If $k$ is allowed to be odd then the conjecture would be false -- for global reasons! There are $p$-newforms
  of level $\Gamma_1(N)\cap\Gamma_0(p)$ and odd weight $k_0$, which automatically have slope
  $(k_0-2)/2\notin\Z$, and in computational examples these forms give rise to both reducible and  irreducible local mod~$p$ representations. The corresponding local $p$-adic Galois representations are now
  semistable rather than crystalline, and depend on an additional parameter, the
  \emph{$\cL$-invariant}; the reduction of the Galois representation depends on
  this $\cL$-invariant in a complicated fashion, see for example the calculations
  of~\cite{MR1944572}. Considering oldforms which are sufficiently $p$-adically
  close to such  newforms (and these will exist by the theory of the eigencurve)
  produces examples of
  $V_{k,a_p}$ with $v(a_p)=(k_0-2)/2$ and $\overline{V}_{k,a_p}$ reducible. If $k$ and $k_0$ are close in weight space then $k$ will also be odd.

The main result of~\cite{MR3081546} determines, for odd~$p$, exactly for which
$a_p$ with $0<v(a_p)<1$ the representation $\overline{V}_{k,a_p}$ is
irreducible; it is necessary that $k\equiv 3\pmod{p-1}$, that $k\ge 2p+1$, and
that $v(a_p)=1/2$, and there are examples for all~$k$ satisfying these conditions.
\end{rem}
\begin{rem} If $p=2$ then the conjecture is also false for the trivial reason
  that if $k\equiv4$ mod~6 then $\overline{V}_{k,0}$ is reducible and hence
  $\overline{V}_{k,a}$ is reducible for $v(a)$ sufficiently large (whether or
  not it is integral) by the main result of~\cite{MR2060368}. In particular,
  the conjecture does not offer a local explanation for the global
  phenomenon that thousands of slopes of cusp forms have been computed for
  $N=1$ and $p=2$, and not a single non-integral one has been found (and
  the conjectures of~\cite{MR2141701} predict that the slopes will all be integral).
\end{rem}
\begin{rem}
  \label{rem: known results on integrality of slopes}Conjecture~\ref{conj:
    integral slopes} is known if $v(a_p)\in (0,1)$, which is the main result
of~\cite{MR2511912}. It is also known if $v(a_p)>\lfloor (k-2)/(p-1)\rfloor$, by
the main result of~\cite{MR2060368}. In the case that $k\le (p^2+1)/2$, it is
expected to follow from work in progress of Yamashita and Yasuda.

The result of~\cite{MR2060368} is proved by constructing an explicit family of
$(\phi,\Gamma)$-modules which are $p$-adically close to the representation
$V_{k,0}$. Since $V_{k,0}$ is induced from a Lubin--Tate character, it has
irreducible reduction if $k$ is not congruent to $1$ modulo $p+1$, and in
particular has irreducible reduction when $p>2$ and $k$ is even, which implies
the result.

In contrast, the papers~\cite{MR2511912,MR3081546} use the $p$-adic local
Langlands correspondence for $\GL_2(\Qp)$ to compute $\overline{V}_{k,a_p}$ more
or less explicitly. Despite the simplicity of the calculations
of~\cite{MR2511912}, which had originally made
us optimistic about the prospects of proving Conjecture~\ref{conj: integral
  slopes} in general, it seems that when $v(a_p)>1$ the calculations involved in
computing $\overline{V}_{k,a_p}$ are very complicated, and without having some
additional structural insight we are pessimistic that Conjecture~\ref{conj:
  integral slopes} can be directly proved by this method.
\end{rem}
In the light of the previous remark, we feel that it is unlikely that
Conjecture~\ref{conj: integral slopes} will be proved without some gaining some
further understanding of why it should be true. We therefore regard the
following question as important.

\begin{question}
  Are there any local or global reasons that we should expect
  Conjecture~\ref{conj: integral slopes} to hold, other than the 
  computational evidence of the second author discussed in~\cite{MR2141701}?
\end{question}

\begin{rem}
  \label{rem: should there be slope conjectures any more generally?}It seems
  unlikely that any analogue of Conjecture~\ref{conj: integral slopes} will hold
  in a more general setting (i.e.\ for higher-dimensional representations of
  $\Gal(\Qpbar/\Qp)$, or for representations of $\Gal(\Qpbar/F)$ of dimension
  $>1$, where $F/\Qp$ is a non-trivial extension). The reason for this is that
  there is no analogue of the fact that $V_{k,a_p}$ is completely determined by
  $k$ and $a_p$; in these more general settings, additional parameters are
  needed to describe the $p$-adic Hodge filtration, and it is highly likely that
  the reduction mod $p$ of the crystalline Galois representations will depend on
  these parameters. (Indeed, as remarked above, this already happens for
  semistable 2-dimensional representations of $\Gal(\Qpbar/\Qp)$.)

For this reason we are sceptical that there is any simple generalisation of the conjectures of
Section~\ref{sec: KB conjectures}, except to the case of Hilbert modular forms
over a totally real field in which $p$ splits completely. For example,
Table~5 in~\cite{MR2452552} and the comments below it show that non-integral slopes appear essentially immediately when one computes with $U(3)$.
\end{rem}

\subsection{The boundary of weight space.}\label{subsec:boundary of weight space}
Perhaps surprisingly, near the boundary of weight space, the combinatorics
of the eigencurve seem to become simpler. For example if $N=1$
and $p=2$ one can compare
Corollary~1 of~\cite{MR2135279} (saying that in weight~0 all overconvergent
slopes are determined by a complicated combinatorial formula) with
Theorem~B of~\cite{MR2135280} (saying that at the boundary
of weight space the slopes form an arithmetic progression). 

 Now let $f$ be a newform of weight $k\geq2$
and level $\Gamma_1(Np^r)$, with $r\geq2$, and
with character whose $p$-part $\chi$ has conductor $p^r$.
For simplicity, fix an isomorphism $\C=\Qpbar$.
Say $f$ has $U_p$-eigenvalue $\alpha$.
One checks that
the associated smooth admissible representation of $\GL_2(\Q_p)$
attached to~$f$ must be principal series associated to two characters
of $\Q_p^\times$, one
unramified (and sending $p$ to $\alpha$) and the other of conductor~$p^r$.
Now say $\rho_f$ is the
$p$-adic Galois representation attached to~$f$. 

By local-global compatibility (the main theorem
of~\cite{MR1465337}),
and the local Langlands correspondence,
the $F$-semisimplified Weil--Deligne representation associated to $\rho_f$
at~$p$ will be the direct sum of two characters, one unramified and the
other of conductor $p^r$. Moreover, the $p$-adic Hodge-theoretic
coincidence still holds: there is at most one possible weakly admissible
filtration on this Weil--Deligne representation with jumps at $0$ and $k-1$,
by Proposition 2.4.5
of~\cite{MR2642406} (or by a direct calculation). 

The resulting
weakly admissible module depends only on $k$, $\alpha$ and $\chi$, and
so we may call its associated Galois representation $V_{k,\alpha,\chi}$;
the local-global assertion is then that this is representation is
the restriction of $\rho_f$ to the absolute Galois group of $\Q_p$.
Let $\overline{V}_{k,\alpha,\chi}$ denote the semisimplification of
the mod~$p$ reduction of~$V_{k,\alpha,\chi}$.
We propose a conjecture which would go some way towards explaining the results
of~\cite{MR2135280}, \cite{roe:3adic}, \cite{MR2434162} and~\cite{MR2944968}.
We write $v_\chi$ for the $p$-adic valuation $v$ on $\Qpbar$
normalised so that the image of $v_\chi$ on $\Qp(\chi)^\times$ is~$\Z$
(so for $p>2$ we have $v_\chi(p)=1/(p-1)p^{r-2}$.)

\begin{conjecture} If $v_\chi(\alpha)\not\in\Z$ then
$\overline{V}_{k,\alpha,\chi}$ is irreducible.
\end{conjecture}

This is a local assertion so does not follow directly from the results
in the global papers cited above. The four papers above prove that
$v_\chi(\alpha)\in\Z$ if $\alpha$ is an eigenvalue of $U_p$
on a space of modular forms of level $2^r$, $3^r$, $5^2$ and $7^2$ respectively;
note that in all these cases, all the local mod~$p$ Galois representations
which show up are reducible locally at~$p$, for global reasons. In fact,
slightly more is true in the special case $p=2$ and $r=2$: in this
case $\Q_p(\chi)=\Q_2$ so the conjecture predicts that if $v(\alpha)\not\in\Z$
then $\overline{V}_{k,\alpha,\chi}$ is irreducible; yet in~\cite{MR2135280}
it is proved that eigenforms of odd weight, level~4, and character of
conductor~4, all have slopes in $2\Z$.

It is furthermore expected that in the global setting the sequence of slopes is
a finite union of arithmetic progressions; see~\cite[Conj.\
1.1]{1407.0279}. Indeed,  a version of this statement
(sufficiently close to the boundary of weight space, in the
setting of the eigenvariety for a definite quaternion algebra
with $p>2$) is proved by Liu--Wan--Xiao in~\cite{1412.2584}.
\section{A strategy to prove Buzzard's conjectures}\label{sec: my
  heuristic}\subsection{}The following strategy for attacking the
conjectures of Section~\ref{sec: KB conjectures} was explained by the second
author to the first author in 2005, and was the motivation for the research
reported on in the papers~\cite{MR2511912,MR3081546} (which we had originally
hoped would result in a proof of Conjecture~\ref{conj: integral slopes}). 

Assume that $p>2$, and fix a continuous odd, irreducible (and thus modular, by
Serre's conjecture), representation
$\rhobar:\Gal(\Qbar/\Q)\to\GL_2(\Fpbar)$. Assume further that $\rhobar$
satisfies the usual Taylor--Wiles condition that
$\rhobar|_{\Gal(\Qbar/\Q(\zeta_p))}$ is irreducible. 

Let $R_{k}^\loc(\rhobar)$ be the (reduced and $p$-torsion free)
universal framed deformation ring for lifts of $\rhobar|_{\Gal(\Qpbar/\Qp)}$
which are crystalline with Hodge--Tate weights $0,k-1$. This connects to the global setting
via the following consequence of the results of~\cite{MR2505297}.

\begin{prop}
  \label{prop: existence of modular lifts hitting a particular
    component}Maintain the assumptions and notation of the previous two
  paragraphs, so that $p>2$, and $\rhobar:\Gal(\Qbar/\Q)\to\GL_2(\Fpbar)$ is a
  continuous, odd, irreducible representation with $\rhobar|_{\Gal(\Qbar/\Q(\zeta_p))}$ irreducible.

Let $N$ be an integer not divisible by $p$ such that $\rhobar$ is modular of level
  $\Gamma_1(N)$. If $p=3$, assume further that $\rhobar|_{\Gal(\Qpbar/\Qp)}$ is not a
  twist of the direct sum of the mod $p$ cyclotomic character and the trivial
  character. Fix an irreducible component of $\Spec
  R_{k}^\loc(\rhobar)[1/p]$. Then there is a newform $f\in
  S_k(\Gamma_1(N),\Qpbar)$ such that $\rhobar_f\cong\rhobar$, and
  $\rho_f|_{\Gal(\Qpbar/\Qp)}$ corresponds to a point of $\Spec
  R_{k}^\loc(\rhobar)[1/p]$ lying on our chosen component.
\end{prop}
\begin{proof}
  This follows almost immediately from the results of~\cite{MR2505297}, exactly
  as in the proof of~\cite[Prop.\ 3.7]{MR2869026}.  (Note that the condition
  that $f$ is a newform of level $\Gamma_1(N)$ can be expressed in terms of the
  conductor of $\rho_f$, and thus in terms of the components of the local
  deformation rings at primes dividing $N$.)

  More precisely, this argument immediately gives the result in the case that
  $\rhobar_f|_{\Gal(\Qpbar/\Qp)}$ is not a twist of an extension of the trivial
  representation by the mod $p$ cyclotomic character. However, this assumption
  on~$\rhobar_f|_{\Gal(\Qpbar/\Qp)}$ is needed only in the proof of~\cite[Cor.\
  2.2.17]{MR2505297}, where this assumption guarantees that the Breuil--M\'ezard
  conjecture holds for $\rhobar_f|_{\Gal(\Qpbar/\Qp)}$ (indeed, the
  Breuil--M\'ezard conjecture is proved under this assumption
  in~\cite{MR2505297}). The Breuil--M\'ezard conjecture is now known
  for $p>2$, except in the case that $p=3$ and $\rhobar|_{\Gal(\Qpbar/\Qp)}$ is
  a twist of the direct sum of the mod $p$ cyclotomic character and the trivial
  character, so the result follows. (The case that $p\ge 3$ and $\rhobar|_{\Gal(\Qpbar/\Qp)}$ is a
  twist of a non-split extension of the trivial character by the mod $p$
  cyclotomic character is treated in~\cite{paskunasBM}, and the case that $p>3$
  and $\rhobar|_{\Gal(\Qpbar/\Qp)}$ is a twist of the direct sum of the mod $p$
  cyclotomic character and the trivial character is proved
  in~\cite{hu2013breuil}.)
\end{proof}
Suppose that $\rhobar|_{\Gal(\Qpbar/\Qp)}$ is reducible, and that
Conjecture~\ref{conj: integral slopes} holds. Consider $a_p$ as a
rigid-analytic function on $\Spec R_{k}^\loc(\rhobar)[1/p]$; since $v(a_p)\in\Z$ by assumption, we
see that $v(a_p)$ is in fact constant on connected (equivalently, irreducible)
components of $\Spec R_{k}^\loc(\rhobar)[1/p]$. 
\begin{cor}
  \label{cor: set of slopes up to multiplicity is determined locally, if they're
    integers}Maintain the assumptions of Proposition~\ref{prop: existence of
    modular lifts hitting a particular component}, and assume further that $\rhobar|_{\Gal(\Qpbar/\Qp)}$ is reducible. Assume Conjecture~\ref{conj:
    integral slopes}. Then the set of slopes (without multiplicities) of $T_p$
  on newforms $f\in S_k(\Gamma_1(N),\Qpbar)$ with $\rhobar_f\cong\rhobar$ is
  determined purely by $k$ and $\rhobar|_{\Gal(\Qpbar/\Qp)}$; more precisely, it
  is the set of slopes of (the crystalline Frobenius of the Galois representations corresponding to)
  components of~$\Spec R_{k}^\loc(\rhobar)[1/p]$.
\end{cor}
\begin{proof}
  This is immediate from Proposition~\ref{prop: existence of modular lifts
    hitting a particular component} (and the discussion in the 
  preceding paragraph).
\end{proof}
\begin{rem}
  \label{rem: things are more complicated in the reducible case}The
  conclusion of Corollary~\ref{cor: set of slopes up to multiplicity is determined locally, if they're
    integers} seems unlikely to hold if~$\rhobar$ is allowed to be
  (globally) reducible; for example, if~$p=2$, it is known that the slopes of all
  cusp forms for~$\SL_2(\Z)$ are at least~$3$, but there are local
  crystalline representations of slope~$1$ (for example the local 2-adic
representation attached to the unique weight~6 level~3 cuspidal eigenform).
We do not know if there is
  any reasonable ``local to global principle'' when~$\rhobar$ is reducible.
\end{rem}

It would be very interesting to be able to have some control on the
multiplicities with which slopes occur in~$S_k(\Gamma_1(N),\Qpbar)$ (for
example, to show that these multiplicities agree for two weights which are
sufficiently $p$-adically close, as predicted by the Gouv\^ea--Mazur
conjecture), but it is not clear to us how such results could be extracted from
the modularity lifting machinery. If all the irreducible components
of~$R_{k}^\loc(\rhobar)$ were regular, it would presumably be possible to use
the argument of Diamond~\cite{MR1440309} to relate the multiplicities of the same slope
in different weights, but we do not expect this to hold in any generality.

Not withstanding this difficulty, one could still hope to prove the conjectures of~\cite{MR2141701}
up to multiplicity. If Conjecture~\ref{conj: integral slopes} were
known, the main obstruction to doing this would be obtaining
a strong local constancy result for slopes as $k$ varies $p$-adically. More
precisely, we would like to prove the following purely local conjecture for some function
$M(n)$ as in~\S\ref{sec: KB conjectures} above.
\begin{conj}\label{conj: local slope constancy}Let $\rbar:\Gal(\Qpbar/\Qp)\to\GL_2(\Fpbar)$ be reducible.
  If $n\ge 0$ is an integer, and $k,k'\ge n+1$ have $k\equiv
  k'\pmod{(p-1)p^{M(n)}}$, then there is a crystalline lift of $\rbar$ with
Hodge--Tate weights $0,k-1$ and slope $n$ if and only if there is a crystalline
lift of $\rbar$ with Hodge--Tate weights $0,k'-1$ and
slope $n$.\end{conj}
It might well be possible to prove a weak result in the direction of
Conjecture~\ref{conj: local slope constancy} by the methods of~\cite{MR2966990}
(more precisely, to prove the conjecture with a much worse bound on $M(n)$ than
would be needed for interesting applications to the conjectures of~\cite{MR2141701}, but
without any assumption on the reducibility of~ $\rbar$).

Corollary~\ref{cor: set of slopes up to multiplicity is determined locally, if they're
    integers} (which shows, granting as always Conjecture~\ref{conj: integral
    slopes}, that the set of slopes which occur globally is the same as the set
  of slopes that occur locally) shows that it would be enough to prove the global version of
  this statement, and it is possible that the methods of~\cite{MR1632794} could allow one to deduce
  a local constancy result where the dependence on $n$ in ``sufficiently close'' is
  quadratic in $n$. (Note that while it is not immediately clear how to
  adapt the methods of~\cite{MR1632794} to allow $\rhobar$ to be fixed, it seems
  plausible that the methods used to prove~\cite[Thm.\ D]{1407.0279} will be able
  to do this.) Note
  again that the computations of~\cite{MR2059481}
(which in particular disprove the original Gouv\^ea--Mazur conjecture) mean that
we cannot expect to deduce Conjecture~\ref{conj: local slope constancy} (for an optimal
function $M(n)$ of the kind suggested by Buzzard's conjectures) from any
global result that does not use the hypothesis that
$\rhobar|_{\Gal(\Qpbar/\Qp)}$ is reducible.

However, it seems plausible to us that a weak local constancy result of this kind, also valid in the case
that $\rhobar|_{\Gal(\Qpbar/\Qp)}$ is irreducible, could be bootstrapped to give
the required strong constancy, provided that Conjecture~\ref{conj: integral
  slopes} is proved. The idea is as follows: under the assumption of
Conjecture~\ref{conj: integral slopes}, $v(a_p)$ is constrained to be an integer
when $\rhobar|_{\Gal(\Qpbar/\Qp)}$ is reducible. If one could prove a result
(with no hypothesis on the reducibility of $\rhobar|_{\Gal(\Qpbar/\Qp)}$)
saying that if $k,k'$ are sufficiently close in weight space, then the small
slopes of crystalline lifts of $\rhobar|_{\Gal(\Qpbar/\Qp)}$ of Hodge--Tate
weights $0,k-1$ and $0,k'-1$ are also close, then the fact that the slopes are
constrained to be integers could then be used to deduce that the slopes are
equal (because two integers which differ by less than $1$ must be equal.)

\bibliographystyle{amsalpha}
\bibliography{slopes}
\end{document}